\documentclass[12pt]{article}

\usepackage{amsmath,amsthm,amssymb}
\usepackage{fullpage}

\newtheorem{defn}{Definition}[section]

\newtheorem{thm}{Theorem}
\newtheorem{lem}{Lemma}
\newtheorem{coro}{Corollary}
\newtheorem{prop}{Proposition}
\theoremstyle{definition}

\newtheorem*{nota}{Notation}

\title{Coideals of block sequences}

\author{\textsc{Jos\'e G. Mijares}\thanks{jmijares@ivic.gob.ve}\\
Instituto Venezolano\\de Investigaciones Cient\'ificas y
\\Escuela de Matem\'atica\\Universidad Central de Venezuela \and
\textsc{Jes\'us Nieto}\thanks{jnieto@usb.ve}\\ Departamento de Formaci\'on General y Ciencias B\'asicas\\ Universidad Sim\'on Bol\'ivar}
\date{}

\begin{document}

\maketitle

\begin{abstract}
We extend the well known notion of \textit{coideal} on $\mathbb{N}$ to families of block sequences on $FIN_k$ and prove that if a coideal of block sequences is \textit{semiselective} and satisfies a local version of Gowers' theorem \cite{Gow} then the local Ramsey property relative to it can be characterized in terms of the abstract Baire property, and the family of all sets having the local Ramsey property relative to one such coideal is closed under the Suslin operation. We also prove that these coideals satisfy a sort of \emph{canonical partition property} in the sense of Taylor \cite{taylor}, L\'opez-Abad \cite{jordi} and Blass \cite{blass}. This results give us an idea of the conditions to be considered in an abstract study of the local Ramsey property in the context of topological Ramsey spaces (see \cite{todo}).
\end{abstract}

\begin{flushleft}


\end{flushleft}

\section{Introduction}
Let $\mathbb{N}$ be the set of nonnegative integers. For a
given $A\subseteq\mathbb{N}$, let $A^{[\infty]} =
\{X\subseteq A: |X| = \infty\}$ and $A^{[<\infty]} =
\{X\subseteq A: |X| < \infty\}$. Consider the sets of the form:
$$[a,A] = \{B\in\mathbb{N}^{[\infty]} : a\sqsubset B\subseteq
a\cup A\}$$
where $a\in\mathbb{N}^{[<\infty]}$, $A\in\mathbb{N}^{[\infty]}$ and $a\sqsubset B$ means that $a$ is
an initial segment of $B$. The relativized version of the
completely Ramsey property (see \cite{galpri}) for subsets of
$\mathbb{N}^{[\infty]}$, known as \textit{local Ramsey property},
is the following:

\medskip

For a family $\mathcal{H}\subseteq\mathbb{N}^{[\infty]}$, a set
$\mathcal{X}\subseteq \mathbb{N}^{[\infty]}$ is said to be
$\mathcal{H}$--\emph{Ramsey} if for every $[a,A]$ with
$A\in\mathcal{H}$ there exists $B\in\mathcal{H}$ with
$[a,B]\subseteq[a,A]$ such that $[a,B]\subseteq\mathcal{X}$ or
$[a,B]\cap\mathcal{X}=\emptyset$. $\mathcal{X}$ is said to be
$\mathcal{H}$--\emph{Ramsey null} if for every $[a,A]$ with
$A\in\mathcal{H}$ there exists $B\in\mathcal{H}$ with
$[a,B]\subseteq[a,A]$ such that $[a,B]\cap\mathcal{X}=\emptyset$.

\medskip

In \cite{mathias}, Mathias introduces the \emph{happy families}
(or \textit{selective coideals}) of subsets of $\mathbb{N}$ and study the
local Ramsey property relative to such families. Then he proves
that the analytic subsets of $\mathbb{N}^{[\infty]}$ are
$\mathcal{U}$--Ramsey when $\mathcal{U}$ is a
Ramsey ultrafilter and generalizes this result for arbitrary happy
families. Farah \cite{farah}
improved this results by introducing the notion of \textit{semiselectivity}
and proving that a coideal is semiselective if and only if
the local Ramsey property is equivalent to a version of the abstract
Baire property relative to that coideal.

\medskip

Let $FIN_k$ be the discretization of the positive part of the unit sphere of the Banach space $c_0$ used by Gowers to study a sort of stability for Lipschitz functions (see \cite{Gow}). In this work, the notions of \textit{coideal} on $\mathbb{N}$ and \textit{semiselectivity} are extended to families of block sequences on $FIN_k$ and it is proven that if a coideal of block sequences is semiselective and satisfies a \textit{local version of Gowers' theorem} \cite{Gow} then the results from \cite{farah} can be translated to the context of $FIN_k$. The structure of this work is as follows: in section
\ref{fink} we present the definition of $FIN_k$
and related notions and state some useful known
results. In section \ref{theory}, the notion of \textit{coideal of block sequences} is introduced, \textit{semiselectivity} and the \textit{Gowers property} for coideals of block sequences are analized, and some examples are given. The corresponding \emph{local Ramsey property} is proven
to be equivalent to a version of the abstract Baire property, when relativized to
a semiselective Gowers coideal of block sequences. In section \ref{suslinop}
we prove that, relative to one such coideal, the family of locally Ramsey sets in this context is closed
under the Suslin operation, showing in this way that this family
includes the analytic sets. In section
\ref{cpp}, we show that every semiselective Gowers coideal satisfies a sort of \emph{canonical
partition property}, in the sense of Taylor \cite{taylor}, L\'opez-Abad
\cite{jordi} and Blass \cite{blass}.

\section{Preliminaries}\label{fink}

Fix an integer $k\geq 1$. Given a function $p\colon
\mathbb{N}\to \{0,1,\dots,k\}$, denote $supp(p)=\{n\colon p(n)
\neq 0\}$ and $rang(p)$ the image set of $p$.
Consider the set
$$FIN_k=\{p\colon\mathbb{N}\to \{0,1,\dots,k\}\colon |supp(p)|<\infty
{\mbox { and }} k\in rang(p)\}.$$
We say that $X=(x_n)_{n\in \mathcal{I}}\subseteq FIN_k$, with
$\mathcal{I}\in \mathcal{P}(\mathbb{N})$ is a \textbf{basic block sequence} if
\begin{center}
$n<m\ \Rightarrow $\ max($supp(x_n))<$\ min($supp(x_m))$
\end{center}
The \emph{length} of $X$, denoted by $|X| $, is the cardinality of
$\mathcal{I}$. For infinite basic block sequences (i.e., basic block
sequences of infinite length) we assume that $\mathcal{I}=\mathbb{N}$.
Define $T\colon FIN_k \to FIN_{k-1}$ by
\begin{center}
$T(p)(n)=$max$\{p(n)-1,0\}$
\end{center}
For $j\in\mathbb{N}$, $T^{(j)}$ is the $j$-th iteration of $T$, i.e.,
$T^{(0)}(p) = p$ and $T^{(j+1)}(p) = T(T^{(j)}(p))$.
Given a basic block sequence $A=(a_n)_{n\in\mathcal{I}}$ we
define $[A]\subseteq FIN_k$ as the set whose elements are all the functions of the form
$$a=T^{(j_0)}(a_{n_0})+T^{(j_1)}(a_{n_1})+\cdots+T^{(j_r)}(a_{n_r})$$
with $n_0<n_1<\cdots<n_r\in \mathcal{I}$, $\{j_0,j_1,\cdots,j_r\}
\subseteq\{0,1,\dots,k\}$, and $j_i=0$ for some $i\in\{0,1,\dots,r\}$.
In this case we say that $a$ is obtained from $A$ by the
\textbf{tetris operation}. Denote by $FIN_k^{[\infty]}$ (resp.
$FIN_k^{[<\infty]}$), the set of infinite (resp. finite) basic block
sequences. Also, denote by $FIN_k^{[n]}$ the set of finite basic block
sequences of length $n$. For $a=(a_1,a_2,\dots,a_n)\in FIN_k^{[n]}$,
let $supp(a) = \cup_{j=1}^nsupp(a_j)$. For $A$, $B\in FIN_k^{[\infty]}$,
define
$$A\leq B \Leftrightarrow A\subseteq [B]$$
If $A=(a_1,a_2,\dots\ )\in FIN_k^{[\infty]}$, for every integer $n\geq 1$,
denote $$A\upharpoonright n : =(a_1,a_2,\dots,a_n)\in FIN_k^{[n]}$$
\noindent and $A\upharpoonright  0 : =\emptyset$. We say that
$a\in FIN_k^{[<\infty]}$ is \emph{compatible} with $A$ (or $A$ is
compatible with $a$) if there exists
$B\leq A$ such that $a=B\upharpoonright n$ for some $n$.
In this case we say that $a$ is an \emph{initial segment} of $B$ and
write $a\sqsubset B$. Denote by $[A]^{[<\infty]}$ (resp. $[A]^{[n]}$)
the set of those members of $FIN_k^{[<\infty]}$ (resp. $FIN_k^{[n]}$)
which are compatibles with $A$.

\bigskip

$FIN_k$ is the discretization of the positive part of the unit sphere of the Banach space $c_0$ used by Gowers \cite{Gow} to study a sort of stability for Lipschitz functions. The following is the combinatorial tool used to prove the main result contained in \cite{Gow}. It will play an important role in our study of the  local Ramsey property in the sequel:

\begin{thm}[Gowers \cite{Gow}]\label{gowers}
Given an integer $r>0$ and
$$f\colon FIN_k\to \{0,1,\dots,r-1\}$$
there exists $A\in FIN_k^{[\infty]}$ such that $f$ is constant on $[A]$.
\end{thm}
\qed

Let $FIN = \mathbb{N}^{[<\infty]}\setminus\{\emptyset\}$. There is an obvious way to identify $FIN$ with $FIN_1$. In this way, for $k=1$ Theorem \ref{gowers} reduces to Hindman's theorem \cite{hindman}.

\medskip

Given $n\in\mathbb{N}$, let $e_n : \mathbb{N} \to \{0, 1, \dots, k\}$
be defined as $e_n(n) = k$ and $e_n(m) = 0$, for every $m\neq n$. It is
clear that $[(e_n)_n] = FIN_k$. So, if $A = (a_n)_n\in FIN_k^{[\infty]}$
then using the canonical isomorphism $\Phi : FIN_k \to [A]$ obtained
by extending the mapping $e_n \mapsto a_n$, the following ``relativized''
version of Gowers' theorem can be proven:

\begin{thm}\label{gowers1}
Given an integer $r>0$, $A\in FIN_k^{[\infty]}$ and
$$f\colon [A]\to \{0,1,\dots,r-1\}$$
 there exists $B\leq A$ such that $f$ is constant on $[B]$.
\end{thm}
\qed

\medskip

For $a=(a_1,\dots, a_n)$, $b=(b_1,\dots,b_m)\in FIN_k^{[<\infty]}$,
write $a< b$ to mean $max(supp(a))<min(supp(b))$. Notice that if $a< b$
then we can build the ``\emph{concatenation}'' $c=a^{\smallfrown}b=(a_1,
\dots,a_n,b_1,\dots,b_m)\in FIN_k^{[<\infty]}$. Define
$$[A]^{[<\infty]}/a=\{b\in[A]^{[<\infty]}\colon a< b\}$$
$$A/a=\{b\in A\colon a< b\}$$ and for every $n\in\mathbb{N}$,
$$A/n=\{b\in A\colon n < min(supp(b))\}$$
Notice that $A/a, A/n\in FIN_k^{[\infty]}$. Also, define the ``Ellentuck
type'' neighborhood
$$[a,A] : =\{B\in FIN_k^{[\infty]}\colon a\sqsubset B {\mbox{ and }}
B/a\subseteq [A]\}$$
Notice that if $a\in[A]^{[<\infty]}$ then
$$[a,A]=\{B\in FIN_k^{[\infty]}\colon a\sqsubset B {\mbox{ and }}
B\leq A\}$$ Also, let $$[a,A]^{[n]} : = \bigcup \{[B]^{[n]} : B\in
[a,A]\}.$$and $$[a,A]^{[<\infty]} = \bigcup_n[a,A]^{[n]}.$$

\section{Coideals of block sequences}\label{theory}

\begin{defn}\label{defncoideal}
We say that ${\cal H}\subseteq FIN_k^{[\infty]}$ is a
\textbf{\emph{coideal of block sequences}} or a \textbf{\emph{coideal}} on $(FIN_k^{[\infty]},\leq)$ if it satisfies
the following:
\begin{enumerate}
\item If $A\leq B$ and $A\in {\cal H}$ then $B\in {\cal H}$.
\item Given $A\in\mathcal{H}$ and a partition $A=B\cup C$, there exists
$D\in \mathcal{H}$ such that $D\leq B$ or $D\leq C$.
\end{enumerate}
\end{defn}

\begin{nota}
For $\mathcal{S}\subseteq FIN_k^{[\infty]}$ and $A\in FIN_k^{[\infty]}$,
denote
$$\mathcal{S}\!\!\upharpoonright\!\!A : = \{B\in \mathcal{S} : B\leq A\}$$
\end{nota}

$FIN_k^{[\infty]}$ is a trivial example of a coideal of block sequences.
To see another example, consider a coideal $H$ on $\mathbb{N}$
and for every $A\in FIN_k^{[\infty]}$ define
$$\mu(A)=\bigcup_{a\in A}\{n\in\mathbb{N}\colon a(n)=k\}$$
Then
$$\mathcal{H}=\{A\in FIN_k^{[\infty]}\colon \mu(A)\in H\}$$
is a coideal of block sequences.

\medskip

\begin{defn}\label{defgowers}
Given $\mathcal{S}\subseteq FIN_k^{[\infty]}$, we say that $\mathcal{S}$ is \textbf{Gowers} or has the \textbf{Gowers property} if
for every integer $r>0$, $A\in \mathcal{H}$ and $f\colon FIN_k\to
\{0,1,\dots,r-1\}$ there exists $B\in \mathcal{S}\!\!\upharpoonright\!\! A$
such that $f$ is constant on $[B]$.
\end{defn}

To give some examples of coideals of block sequences having the Gowers property we shall use the following
consequence of Gowers' theorem.

\begin{prop}\label{pseudo-frechet}
Let $\mathcal{H}\subseteq FIN_k^{[\infty]}$ be such that
\begin{enumerate}
\item If $A\leq B$ and $A\in \mathcal{H}$ then $B\in \mathcal{H}$.
\item $\forall A\in \mathcal{H}$ $\exists B\leq A$ $([0,B]\subseteq
\mathcal{H})$.
\end{enumerate}
Then $\mathcal{H}$ is a Gowers coideal of block sequences.
\end{prop}

\begin{proof}
Given $r$, $A$ and $f$ as in the definition \ref{defgowers}, consider
$B$ as in part (2) of the hypothesis and apply Theorem \ref{gowers}
to $r$, $B$ and $f$ to obtain $B'\leq B$ such that $f$ is constant
on $[B']$. Since $B'\in[0,B]\subseteq\mathcal{H}$, we have that
$\mathcal{H}$ is Gowers. It is clear that part (2) of the definition
of coideal follows from $\mathcal{H}$ being Gowers. This concludes the
proof.
\end{proof}

Now we give some examples of Gowers coideals. Fix a nonempty
$\mathcal{A}\subseteq FIN_k^{[\infty]}$ and define
$$\mathcal{A}^{\top}=\{B\in FIN_k^{[\infty]}\colon \exists A\in\mathcal{A}
\ \exists C\leq B(C\leq A)\}$$
It is clear that $\mathcal{A}^{\top}$ satisfies (1) and (2) of
proposition \ref{pseudo-frechet}, hence $\mathcal{A}^{\top}$ is a Gowers coideal of block sequences, for every choice of $\mathcal{A}$. Another feature
of these examples is a property which is analogous to \emph{semiselectivity} for coideals on $\mathbb{N}$.

\medskip

A set $\mathcal{D}\subseteq FIN_k^{[\infty]}$ is
\textbf{dense open} if it satisfies:
\begin{itemize}
\item[I)] If $B\in\mathcal{D}$ and $A\leq B$ then $A\in\mathcal{D}$.
\item[II)] $\forall A\in FIN_k^{[\infty]}$ $\exists B\in\mathcal{D}$
$(B\leq A)$.
\end{itemize}

\medskip

Also, if $(A_n)_{n\geq1}$ is a decreasing sequence in $(FIN_k^{[\infty]},\leq)$, we say that $B\in FIN_k^{[\infty]}$ is a {\bf diagonalization} of $(A_n)_{n\geq 1}$ if $B/b\leq A_n$, for every $b\in[B]$ with $n=max(supp(b))$. Notice that for such $B$ we have $[b,B]\subseteq[b,A_n]$, for every
$b\in[B]^{[<\infty]}$ with $n=max(supp(b))$.

\begin{defn}
We say that ${\cal H}\subseteq FIN_k^{[\infty]}$ is \textbf{semiselective}
if given $A\in \mathcal{H}$ and a sequence $(\mathcal{D}_n)_n\subseteq \mathcal{H}$ of dense open sets, there exists a decreasing
$(A_n)_n$ with $A_n\in\mathcal{D}_n$ for all $n$ and $B\in
\mathcal{H}\!\!\upharpoonright\!\! A$ such that $B$ diagonalizes $(A_n)_n$.
Also, we say that $B$ is a diagonalization of the sequence $(\mathcal{D}_n)_n$.
(In other words, the set of diagonalizations of $(\mathcal{D}_n)_n$
is \emph{dense} in $(\mathcal{H},\leq)$).
\end{defn}

Let us see that every $\mathcal{A}^{\top}$ as defined above
is semiselective: Let $(\mathcal{D}_n)_n\subseteq \mathcal{A}^{\top}$ be a
sequence of dense open sets. Given $A\in \mathcal{A}^{\top}$,
fix $B\leq A$ such that $[0,B]\subseteq \mathcal{A}^{\top}$. Using the density of each $\mathcal{D}_n$, it is easy to choose a decreasing sequence $(A_n)_n$ with $A_n\in\mathcal{D}_n$ with $A_0\leq B$. If we pick $c_n\in [A_n]$ with $c_n < c_{n+1}$ then $C = \{c_0, c_1, \dots\}\leq B$ and diagonalizes $(\mathcal{D}_n)_n$. But $C\in\mathcal{A}^{\top}$, so we are done.

\medskip

We have given a \emph{scheme} of examples of semiselective Gowers coideals.
In next section we shall see that this type of coideals is very
convenient for our study.

\section{The Ramsey property}\label{ramsey-prop}

For the next two definitions, let ${\cal H}$ be a coideal on $(FIN_k^{[\infty]},\leq)$.

\begin{defn} $\mathcal{X}\subseteq FIN_k^{[\infty]}$ is
$\mathcal{H}$--\textbf{Ramsey}
if given $A\in \mathcal{H}$ and $a\in FIN_k^{[<\infty]}$
 there exists $B\in [a,A]\cap {\cal H}$ such that $[a,B]\subseteq \mathcal{X}$ or $[a,B]\cap\mathcal{X}=
\emptyset$. If for every $A\in \mathcal{H}$ and $a\in FIN_k^{[<\infty]}$
there exists $B\in [a,A]\cap {\cal H}$ such that $[a,B]\cap\mathcal{X}=\emptyset$, we say that
$\mathcal{X}$ is $\mathcal{H}$--\textbf{Ramsey null}.
\end{defn}

\begin{defn} $\mathcal{X}\subseteq FIN_k^{[\infty]}$ is
$\mathcal{H}$--\textbf{Baire} if given $A\in \mathcal{H}$
and $a\in FIN_k^{[<\infty]}$ there exists $[b,B]\subseteq[a,A]$,
with $B\in{\cal H}$, such that
$[b,B]\subseteq \mathcal{X}$ or $[b,B]\cap\mathcal{X}=\emptyset$.
If for every $A\in \mathcal{H}$ and $a\in FIN_k^{[<\infty]}$
there exists $[b,B]\subseteq[a,A]$ with $B\in \mathcal{H}$, such that
$[b,B]\cap\mathcal{X}=\emptyset$, we say that $\mathcal{X}$ is
$\mathcal{H}$--\textbf{nowhere dense}.
\end{defn}

The main result of this work is the following.

\begin{thm}\label{baire-ramsey}
If $\mathcal{H}$ is a semiselective Gowers coideal of block sequences then, for every $\mathcal{X}\subseteq FIN_k^{[\infty]}$ :
\begin{enumerate}
\item $\mathcal{X}$ is $\mathcal{H}$--Ramsey
iff $\mathcal{X}$ is ${\cal H}$--Baire.
\item $\mathcal{X}$ is $\mathcal{H}$--Ramsey null
iff $\mathcal{X}$ is $\mathcal{H}$-- nowhere dense.
\end{enumerate}

\end{thm}

One of the consequence of Theorem \ref{baire-ramsey} is that analytic subsets of $FIN_k^{[\infty]}$ are $\mathcal{H}$--Ramsey, for every semiselective Gowers coideal of block sequences $\mathcal{H}$. Here we are viewing $FIN_k^{[\infty]}$ as a subspace of $FIN_k^{\mathbb{N}}$ with the (metric) product topology, regarding $FIN_k$ as a discrete space; we will see more about this in Section \ref{suslinop} (see Theorem \ref{ana-fink} below). In Section \ref{cpp} we will use this fact to show that semiselective Gowers coideals satisfy a canonical partition property (see Theorem \ref{semisel-cpp}) similar to the one satisfied by \textit{stable ordered-union ultrafilters} in \cite{blass} and related to the generalization of Taylor's theorem \cite{taylor} due to L\'opez-Abad \cite{jordi}.

\medskip

Before showing our proof of Theorem \ref{baire-ramsey}, we will need to prove a version of the semiselective Galvin's lemma (see \cite{farah} and \cite{galvin}), for the context of $FIN_k$.

\medskip

Let us consider the following \emph{combinatorial forcing}. Fix a coideal $\mathcal{H}$ and ${\cal F}\subseteq FIN_k^{[<\infty]}$. We say that
$B\in {\cal H}$ \textbf{accepts} $a\in
FIN_k^{[<\infty]}$ if for every $B'\in [a,B]$ there exists $b\in\mathcal{F}$
such that $b\sqsubset B'$. $B$ \textbf{rejects} $a$ if no member
of $[a,B]\cap{\cal H}$ accepts $a$; and $B$ \textbf{decides} $a$
if $B$ accepts or rejects $a$. This combinatorial forcing has the
following features:

\begin{lem}\label{features}
If $\mathcal{H}$ is Gowers then:
\begin{enumerate}
\item If $B$ accepts (rejects) $a$, then every $B'\in \mathcal{H}
\!\!\upharpoonright \!\!B$ accepts (rejects) $a$.
\item Given $B\in \mathcal{H}$ and $a\in
FIN_k^{[<\infty]}$ there exists $B'\in \mathcal{H}\!\!\upharpoonright\!\! B$
which decides $a$.
\item If $B$ accepts $a$ then $B$ accepts every $b\in [a,B]^{[|a|+1]}$.
\item If $B$ rejects $a$ then there exists $B'\in [a,B]\cap\mathcal{H}$
such that $B$ does not accept any $b\in [a,B']^{[|a|+1]}$.
\end{enumerate}

\end{lem}
\begin{proof}
1--3 follow from the definitions. To prove 4, let
$$\mathcal{O}=\{b\in FIN_k^{[|a|+1]}\colon B {\mbox { accepts }} b\}$$
Notice that, since $\mathcal{H}$ is Gowers
there exists $B'\in[a,B]\cap {\cal H}$ such that
$$[a,B']^{[|a|+1]}\subseteq {\cal O} \mbox{ or }
[a,B']^{[|a|+1]}\subseteq {\cal O}^{\,c}$$
Suppose that
$[a,B']^{[|a|+1]}\subseteq {\cal O}$. Since
$$[a,B']=\bigcup_{b\in [a,B']^{[|a|+1]}}[b,B']$$
we have that $B'$ accepts $a$, which contradicts that $B$
rejects $a$. Hence, $[a,B']^{[|a|+1]}\subseteq {\cal O}^{\,c}$ and therefore, $B$ does not accept any $b\in [a,B']^{[|a|+1]}$.
\end{proof}

\begin{lem}\label{decides}
If $\mathcal{H}$ is semiselective and Gowers then for every $A\in\mathcal{H}$
there exists $B\in {\cal H} \!\!\upharpoonright \!\!A$ which decides
every $b\in [B]^{[<\infty]}$.
\end{lem}

\begin{proof}
Let $A\in\mathcal{H}$ be given. For every $a\in FIN_k^{[<\infty]}$ define
$${\cal D}_a=\{C\in\mathcal{H}\colon C \mbox{ decides }a\}$$
By lemma \ref{features}, ${\cal D}_a$ is dense open in $(\mathcal{H},\leq)$. For every $n\in\mathbb{N}$, let $$\mathcal{D}_n=
\cap\{\mathcal{D}_a\colon max(supp(a))\leq n\}.$$ Then every $\mathcal{D}_n$
is also dense open in $(\mathcal{H},\leq)$. Let $B\in \mathcal{H}\!\!\upharpoonright\!\!A$ be a diagonalization of
$(\mathcal{D}_n)_n$. Then, for every $b\in[B]^{[<\infty]}$ with
$max(supp(b))=n$ there exists $D\in \mathcal{D}_n$ such that $[b,B]
\subseteq[b,D]$. Thus, $B$ decides $b$.
\end{proof}

\begin{thm}[Semiselective Galvin's lemma for block sequences]\label{galvinlocal}
Given a semiselective Gowers coideal $\mathcal{H}$, $A\in{\cal H}$
and ${\cal F}\subseteq FIN_k^{[<\infty]}$, there exists $B\in {\cal H}\!\!
\upharpoonright\!\!A$ such that:
\begin{enumerate}
\item $[B]^{[<\infty]}\cap{\cal F}=\emptyset$, or
\item $\forall C\in [\emptyset,B]$ $(\exists \ a\in{\cal F})$
$(a\sqsubset C)$.
\end{enumerate}
\end{thm}

\begin{proof}
Fix a semiselective Gowers coideal $\mathcal{H}$, $A\in{\cal H}$
and ${\cal F}\subseteq FIN_k^{[<\infty]}$. Consider the combinatorial forcing defined above and let $B\in {\cal H} \!\!\upharpoonright \!\!A$ be as in lemma \ref{decides}. If $B$ accepts $\emptyset$ then Part 2 of the theorem holds, so suppose that $B$ rejects $\emptyset$.
For every  $a\in [B]^{[<\infty]}$ define
$${\cal D}_a=\{C\in{\cal H}\!\!\upharpoonright\!\!B\colon C
{\mbox{ rejects every }}b\in r_{|a|+1}([a,C])\},$$if $B$ rejects $a$;
and $\mathcal{D}_a={\cal H}\!\!\upharpoonright\!\!B$, otherwise.
Then every $\mathcal{D}_a$ is dense open in $(\mathcal{H}\!\!\upharpoonright\!\!B,\leq)$. Let $B_1$
be a diagonalization of $(\mathcal{D}_a)_a$ in
$\mathcal{H}\!\!\upharpoonright\!\!B$. Let us see that $B_1$ rejects every
$a\in [B_1]^{[<\infty]}$:

\medskip

In fact, $B_1$ rejects $\emptyset$
since $B$ rejects $\emptyset$. Assume that $B_1$ rejects every
$a\in [B_1]^{[n]}$ and consider $b\in [B_1]^{[n+1]}$. Then
$b\in [a,B_1]^{[n+1]}$ for some $a\in [B_1]^{[n]}$. Let $D\in
\mathcal{D}_a$ be such that $B_1/a\leq D$. Since
$[b,B_1]\subseteq[b,D]$ and $D$ rejects $b$ ($\mathcal{D}_a\neq
{\cal H}\!\!\upharpoonright\!\!B$ because $B_1$ rejects $a$), we have
that $B_1$ rejects $b$. Thus, no member of $[B_1]^{[<\infty]}$ is in
$\mathcal{F}$ and so Part 1 of the theorem holds.
\end{proof}

\begin{coro}[Galvn's lemma for block sequences]
Given ${\cal F}\subseteq FIN_k^{[<\infty]}$ and $A\in FIN_k^{[\infty]}$, there exists $B\leq A$ such that:
\begin{enumerate}
\item $[B]^{[<\infty]}\cap{\cal F}=\emptyset$, or
\item $\forall C\in [\emptyset,B]$ $(\exists \ a\in{\cal F})$
$(a\sqsubset C)$.
\end{enumerate}
\end{coro}
\qed

\bigskip

It is not hard to prove the following generalization of Theorem
\ref{galvinlocal}:

\begin{thm}\label{galvinlocal2}
Given a semiselective Gowers coideal $\mathcal{H}\subseteq
FIN_k^{[\infty]}$, $A\in\mathcal{H}$, $\mathcal{F}\subseteq FIN_k^{[<\infty]}$
and $a\in FIN_k^{[<\infty]}$, there exists $B\in [a,A]\cap\mathcal{H}$
such that:
\begin{enumerate}
\item $[a,B]^{[<\infty]}\cap{\cal F}=\emptyset$, or
\item $\forall C\in [a,B]$ $(\exists \ b\in{\cal F})$ $(b\sqsubset C)$.
\end{enumerate}
\end{thm}
\qed

\bigskip

We are now ready to prove our main result.

\begin{proof} [\textbf{Proof of Theorem \ref{baire-ramsey}}]

Let ${\cal H}$ be a semiselective Gowers coideal. To obtain Part 1 we only have to prove the implication
$(\Leftarrow)$. Let $\mathcal{X}$ be ${\cal H}$--Baire in
$FIN_k^{[\infty]}$. Fix $[a,A]$ with $A\in {\cal H}$ and define
$${\cal F}_0=\{b\in FIN_k^{[\infty]}\colon a\sqsubseteq b {\mbox{ and
}}[b,A]\subseteq \mathcal{X}\}$$			
Consider $B_0\in [a,A]\cap\mathcal{H}$ as in Theorem \ref{galvinlocal2} when applied to
$\mathcal{F}_0$, $a$ and $B$. If Part 2 of Theorem \ref{galvinlocal2}
holds then $[a,B_0]\subseteq \mathcal{X}$ and we are done. Otherwise
define
$${\cal F}_1=\{b\in[B_0]^{[<\infty]}\colon a\sqsubseteq b {\mbox{ and
}}[b,B_0]\cap\mathcal{X}=\emptyset\}$$
Consider $B_1\in [a,B_0]\cap\mathcal{H}$ as in Theorem \ref{galvinlocal2} when applied to
$\mathcal{F}_1$, $a$ and $B_0$. If Part 2 of Theorem
\ref{galvinlocal2} holds for $B_1$ then $[a,B_1]\cap\mathcal{X} =
\emptyset$ and we are done. We claim that Part 1
of Theorem \ref{galvinlocal2} is not possible for $B_1$: otherwise, we could find $B_2\in [a,B_1]\cap\mathcal{H}$ as in Theorem \ref{galvinlocal2}
when applied to $\mathcal{F}_0\cup\mathcal{F}_1$, $a$ and $B_1$ which would necessarily satisfy $[B_2]^{[<\infty]}\cap(\mathcal{F}_0\cup\mathcal{F}_1) = \emptyset$. But this would contradict that $\mathcal{X}$
is $\mathcal{H}$--Baire.

\medskip

To obtain Part 2, again, we only have to prove the implication $(\Leftarrow)$. But it follows easily from Part 1 because any set which is ${\cal H}$-- nowhere dense and ${\cal H}$--Ramsey must necessarily be ${\cal H}$--Ramsey null. This concludes the proof of Theorem \ref{baire-ramsey}.
\end{proof}

\bigskip

The basic sets of the metric topology of $FIN_k^{[\infty]}$ as a subspace of
$FIN_k^{\mathbb{N}}$ are of the form  $$[b]=\{A\in FIN_k^{[\infty]}\colon b\sqsubset A\},$$
with $b\in FIN_k^{[<\infty]}$. As another consequence of Theorem
\ref{galvinlocal}, we have the following:

\begin{coro}\label{metricosfink}
If $\mathcal{H}$ is a semiselective Gowers coideal of block sequences
then every metric open subset of $FIN_k^{[\infty]}$ is
$\mathcal{H}$--Ramsey.
\end{coro}

\begin{proof}
Let $\mathcal{X}\subseteq FIN_k^{[\infty]}$ be open
and fix $[a,A]$. Without loss of generality we can assume that
$a=\emptyset$. There exists $\mathcal{F}\subseteq FIN_k^{[<\infty]}$
such that $\mathcal{X}=\bigcup_{b\in \mathcal{F}}[b]$. Let $B\leq A$
be as in Theorem \ref{galvinlocal}. If parte 1 from the theorem
\ref{galvinlocal} holds then $[0,B]\subseteq\mathcal{X}^c$. If part
2 from the Theorem \ref{galvinlocal} holds, then $[0,B]\subseteq
\mathcal{X}$.
\end{proof}

In the next section we will see that in fact every analytic subset of $FIN_k^{[\infty]}$ is
$\mathcal{H}$--Ramsey.

\bigskip

We finish this section by proving from Theorem \ref{galvinlocal} the following local version of a generalization of Gowers' theorem \cite{Gow} due to Todorcevic's \cite{todo2}:

\begin{thm}\label{ramsey2}
Assume that ${\cal H}\subseteq FIN_k^{[\infty]}$ is a
semiselective Gowers coideal and $n\in\mathbb{N}$. Then, for every $r\in \mathbb{N}$, $r>1$, and every
$f\colon FIN_k^{[n]}\to \{0,1, \dots, r-1\}$ and $A\in {\cal H}$, there exists
$B\in \mathcal{H}\!\!\upharpoonright \!\!A$ such that $f$ is constant
on $[B]^{[n]}$.
\end{thm}

\begin{proof}
For the case $r=2$, let $\mathcal{F}=f^{-1}(\{0\})$ and consider $B\in \mathcal{H}\!\!\upharpoonright \!\!A$ as in
Theorem \ref{galvinlocal}. For $r>2$, the required result follows by induction on $r$.
\end{proof}

\begin{defn}
A coideal $\mathcal{H}$ on $(FIN_k^{[\infty]},\leq)$ is ${\bf Ramsey}$ if for every
$\mathcal{X}\subseteq FIN_k^{[2]}$, there exists $B\in {\cal H}$ such
that $[B]^{[2]}\subseteq\mathcal{X}$ or $[B]^{[2]}\cap\mathcal{X}
= \emptyset$.
\end{defn}

So Theorem \ref{ramsey2} says that every semiselective Gowers
coideal is a Ramsey coideal.

\section{The Suslin operation}\label{suslinop}

Recall that given a set $X$ and a family ${\cal P}$ of
subsets of $X$, two subsets $A$, $B$ of $X$ are said to be
\emph{compatible} with respect to ${\cal P}$ if there exists $C\in
{\cal P}$ such that $C\subseteq A\cap B$. The family ${\cal P}$ is
said to be \emph{M-like} if given ${\cal Q}\subseteq {\cal P}$ with $|{\cal Q}|<|{\cal P}|$,
every member of ${\cal P}$ which is not compatible with any
member of ${\cal Q}$ is compatible with $X\setminus \bigcup
{\cal Q}$. Also, recall that a $\sigma$-algebra ${\cal A}$ of subsets of $X$
together with a $\sigma$-ideal ${\cal A}_0\subseteq{\cal A}$
is a \emph{Marczewski pair} if for every $A\subseteq X$ there
exists $\Phi(A)\in {\cal A}$ such that $A\subseteq \Phi(A)$ and
for every $B\subseteq \Phi(A)\setminus A$, $B\in{\cal A}\Rightarrow
B\in {\cal A}_0$.

\medskip

The goal of this section is to show that the family of
${\cal H}$--Ramsey subsets of $FIN_k^{[\infty]}$ is closed under
the Suslin operation, whenever ${\cal H}$ is a semiselective Gowers coideal of block sequences. Given a
family $(\mathcal{X}_a)_{a\in FIN_k^{[<\infty]}}$
of subsets of $FIN_k^{[\infty]}$, the result of
applying the Suslin operation to this family is:
$$\bigcup_{A\in FIN_k^{[\infty]}}\bigcap_{n\in\mathbb{N}}
\mathcal{X}_{A\upharpoonright n}$$

\medskip

The following is a well known fact:

\begin{thm}[Marczewski]\label{marcz}
Every $\sigma$-algebra of sets which together with a $\sigma$-ideal
is a Marczeswki pair, is closed under the Suslin operation.
\end{thm}
\qed

The following proposition shows that the family $\mathcal{R}
(\mathcal{H})$ of ${\cal H}$--Ramsey
subsets of ${FIN_k^{[\infty]}}$ is a $\sigma$-algebra and the
collection $\mathcal{R}_0(\mathcal{H})$ of ${\cal H}$--Ramsey null
subsets of ${FIN_k^{[\infty]}}$ is a $\sigma$-ideal  of it.

\begin{prop}\label{sigmaideal}
If ${\cal H}\subseteq FIN_k^{[\infty]}$ is a semiselective Gowers
coideal of block sequences then the families of ${\cal H}$--Ramsey and
${\cal H}$--Ramsey null subsets of $FIN_k^{[\infty]}$ are closed under
countable union.
\end{prop}

\begin{proof}  Fix $[a,A]$ with $A\in{\cal H}$. Again, we will suppose that $a=\emptyset$.

\medskip

Suppose that $({\cal X}_n)_{n\geq1}$ is a
sequence of ${\cal H}$--Ramsey null subsets of $FIN_k^{[\infty]}$. We can and will also assume that $\mathcal{X}_{n}\subseteq \mathcal{X}_{n+1}$ for all $n$
without loss of generality. For every $b\in FIN_k^{[<\infty]}$ define
$$\mathcal{D}_b=\{C\in \mathcal{H}\colon [b,C]\cap \mathcal{X}_n
=\emptyset \mbox{ for all }n\leq|b|\}$$
Every such $\mathcal{D}_b$ is dense open in $({\cal H},\leq)$. For every $n\in\mathbb{N}$, let $$\mathcal{D}_n=\bigcap\{\mathcal{D}_b
\colon max(supp(b))\leq n \}.$$ Then every $\mathcal{D}_n$ is dense open too.
Let $B\in \mathcal{H}\!\!\upharpoonright\!\!A$ be a diagonalization
of $(\mathcal{D}_n)_n$. Then $[0,B]\cap\bigcup_n\mathcal{X}_n = \emptyset$.

\medskip

Now, suppose that $({\cal X}_n)_{n\geq1}$ is a sequence of
${\cal H}$--Ramsey subsets of $FIN_k^{[\infty]}$.
If there exists $B\in{\cal H}\!\!\upharpoonright \!\!A$ such that
$[0,B]\subseteq {\cal X}_n$ for some $n$, we are done. Otherwise,
using an argument similar to the one above, we prove that $\bigcup {\cal X}_n$
 is ${\cal H}$--Ramsey null.
\end{proof}

Given a semiselective Gowers coideal ${\cal H}\subseteq FIN_k^{[\infty]}$,
in order to show that ($\mathcal{R}(\mathcal{H})$,$\mathcal{R}_0(\mathcal{H})$)
forms a Marczeswki pair it is sufficient to prove the following
(see \cite{Mor}, \cite{pawl} or \cite{farah}):

\begin{prop}\label{mlike}
Let ${\cal H}$ be a semiselective Gowers coideal of block sequences.
Assuming \textbf{CH}, the family
$$Exp({\cal H}) : = \{[a,A] : a\in FIN_k^{[\infty]}, A\in{\cal H}\}$$
 is $M$-like.
\end{prop}

\begin{proof} Consider ${\cal B}\subseteq Exp({\cal H})$ with
 $|{\cal B}|<|Exp({\cal H})|=2^{\aleph_0}$ and suppose that
 $[a,A]\in Exp({\cal H})$ and $[a,A]$ is not compatible with any member of ${\cal B}$, i. e.
 for every $Y\in {\cal B}$, $Y\cap [a,A]$ does not contain any
 member of $Exp({\cal H})$. We claim that $[a,A]$ is compatible
  with $FIN_k^{[\infty]}\smallsetminus \bigcup {\cal B}$. In fact:

By proposition \ref{sigmaideal}$,  \bigcup {\cal B}$ is
${\cal H}$--Ramsey. So, there exist  $B\in {\cal H}\upharpoonright A$ such that:
\begin{enumerate}
\item $[a,B]\subseteq \bigcup {\cal B}$ or
\item $[a,B]\subseteq {FIN_k^{[\infty]}}\smallsetminus \bigcup {\cal B}$
\end{enumerate}

Alternative 1 is not possible because $[a,A]$ is not compatible with any member
of ${\cal B}$. This completes the proof.
\end{proof}

\begin{coro}
Assuming \textbf{CH}, if $\mathcal{H}$ is a semiselective Gowers coideal on
$(FIN_k^{[\infty]}, \leq)$ then ($\mathcal{R}(\mathcal{H})$,$\mathcal{R}_0
(\mathcal{H})$) forms a Marczeswki pair.
\end{coro}
\qed

Now we use the following result due to Platek \cite{platek}:

\begin{thm}\label{platek}
The use of \textbf{CH} can be eliminated from the proof of any
statement involving only quantification over the reals and possibly
some fixed set of reals as a predicate.
\end{thm}
\qed

Since the statement
\begin{center}
``$\bigcup_{A\in FIN_k^{[\infty]}}\bigcap_{n\in\mathbb{N}}
\mathcal{X}_{A\upharpoonright n}$
is not $\mathcal{H}$--Ramsey''
\end{center}
is false under \textbf{CH} by Theorem \ref{marcz} and it has the
form required  in Theorem \ref{platek}, we have the following:

\begin{coro}
If $\mathcal{H}$ is a semiselective
Gowers coideal of block sequences then the family of ${\cal H}$--Ramsey
subsets of ${FIN_k^{[\infty]}}$ is closed
under the Suslin operation,
\end{coro}
\qed

\begin{coro}\label{ana-fink}
If $\mathcal{H}$ is a semiselective
Gowers coideal of block sequences then every metric analytic subset of
${FIN_k^{[\infty]}}$ is ${\cal H}$--Ramsey.
\end{coro}
\qed

\section{Canonical partition property}\label{cpp}

On $FIN$ (i.e., $FIN_1$), consider the equivalence relations $min$, $max$, $(min,max)$,
$=$ and $FIN^2$, defined by
$$(s,t)\in min\Leftrightarrow min(a)= min(b)$$
$$(s,t)\in max\Leftrightarrow max(a) = max(b)$$
$$(s,t)\in (min,max)\Leftrightarrow min(a)= min(b) \mbox{ and }max(a)
= max(b)$$
``$=$'' and $FIN^2$ are the trivial relations. List
$$\mathcal{R}_1=\{min,max,(min,max) ,=,FIN^2\}$$

In \cite{taylor} the following was proven:

\begin{thm}[Taylor \cite{taylor}]\label{cpp-fin}
For every equivalence relation $R$ on $FIN$ there exists $A\in
FIN^{[\infty]}$ such that the restriction of $R$ to $[A]$ coincides
with one of the members of $\mathcal{R}_1$.
\end{thm}
\qed

The members of $R_1$ are known as \emph{canonical relations} on $FIN$.

\medskip

For $k > 1$ the corresponding list of canonical equivalence relations is longer. To give the list we need some definitions, taken from \cite{jordi}.

\begin{defn}
For $i\leq k$ define the maps $min_i$, $max_i\colon FIN_k\to \mathbb{N}$
by $min_i(a)=$min $\{n\colon a(n)=i\}$ and 0 if $i\not \in rank(a)$,
$max_i(a)=$max$\{n\colon a(n)=i\}$ and 0 if $i\not \in rank(a)$.
We say that $a\in FIN_k$ is a \textbf{system of staircase} (\textbf{sos} in short)
if it satisfies
\begin{itemize}
\item[a)] $rank(a)=\{0,1,\dots,k\}$
\item[b)] $min_i(a)<min_j(a)<max_j(a)<min_j(a)$, for $i<j\leq k$.
\item[c)] For every $1\leq i\leq k$
$$rank(a\upharpoonright[min_{i-1}(a),min_i(a)])=\{0,1,\dots,i\}$$
$$rank(a\upharpoonright[max_i(a),max_{i-1}(a)])=\{0,1,\dots,i\}$$
$$rank(a\upharpoonright[min_k(a),max_k(a)])=\{0,1,\dots,k\}$$
\end{itemize}
We say that $A=(a_1,a_2,\dots)\in FIN_k^{[\infty]}$ is a sos if every $a_j$ is a sos.
\end{defn}

\medskip

\begin{defn}
Let $R$ be an equivalence relation on $FIN_k$. Given $A\in FIN_k^{[\infty]}$, we say that $R$ is \textbf{canonical in }
$[A]$ if for every sos $B\leq A$ one of the following holds:
\begin{itemize}
\item $\forall a$, $b\in [B]$ $(a,b)\in R$ or
\item $\forall a$, $b\in [B]$ $(a,b)\not\in R$.
\end{itemize}
If $R$ is canonical in $FIN_k$ we say that $R$ is \textbf{canonical}.
\end{defn}

Define the equivalence relations $min_k$, $max_k$, $(min,max)_k$ on $FIN_k$ by
$$(a,b)\in min_k\Leftrightarrow min_k(a)=min_k(b)$$
$$(a,b)\in max_k\Leftrightarrow max_k(a)=max_k(b)$$
$$(min,max)_k = min_k\cap max_k$$
Analogously, define the relations $min_i$, $max_i$ and $(min,max)_i$, for
$1\leq i\leq k-1$. All of these
are examples of canonical relations. In \cite{jordi}, it has been proven
the following generalization of Theorem \ref{cpp-fin}:

\begin{thm}[L\'opez--Abad \cite{jordi}]\label{cpp-fink}
There exists a finite collection
$$\mathcal{R}_k=\{R_1,R_2,\dots,R_{t_k}\}$$ of canonical equivalence
relations on $FIN_k$ such that for every equivalence relation $R$ on
$FIN_k$ there exist $m\in\{1,2,\dots,t_k\}$ and an sos $A\in
FIN_k^{[\infty]}$ such that the restriction of $R$ to $[A]$ coincides with $R_m$.
\end{thm}
\qed

Theorem \ref{cpp-fink} has the following \emph{relativized} version:

\begin{thm}[L\'opez--Abad \cite{jordi}]\label{rcpp-fink}
There exists a finite collection
$$\mathcal{R}_k=\{R_1,R_2,\dots,R_{t_k}\}$$ of canonical equivalence
relations on $FIN_k$ such that for every
$A\in FIN_k^{[\infty]}$ and every equivalence relation $R$ on $[A]$
there exist $m\in\{1,2,\dots,t_k\}$ and an sos $B\leq A$ such that
the restriction of $R$ to $[B]$ coincides with $R_m$.
\end{thm}
\qed

$t_k$ is given by the following:
$$t_k=(k!e_k(1))^2+k(k!e_k(1)-(k-1)!e_{k-1}(1))^2$$
where, for every $n$, $e_n(1)=\sum_{j=0}^n\frac{1}{j!}$.

\medskip

Now, fix $\mathcal{R}_k=\{R_1,R_2,\dots,R_{t_k}\}$ as in Theorem \ref{rcpp-fink}.

\begin{defn}\label{def-cpp}
A coideal $\mathcal{H}$ on $FIN_k$ is said to have
the \textbf{canonical partition property} if for every equivalence
relation $R$ on $FIN_k$ there exist $m\in\{1,2,\dots,t_k\}$ and
$A\in \mathcal{H}$ such that the restriction of $R$ to $[A]$
coincides with $R_m$.
\end{defn}

We conclude with the following

\begin{thm}\label{semisel-cpp}
If $\mathcal{H}$ is a semiselective Gowers coideal of block sequences
then it has the canonical partition property.
\end{thm}

\begin{proof}
Let $R$ be an equivalence relation on $FIN_k$, $\mathcal{R}_k$ as in
Theorem \ref{cpp-fink} and let $$\mathcal{X}=\{A\in FIN_k^{[\infty]}\colon
\exists m\leq t_k\ (R\upharpoonright [A]=R_m)\}$$

Then $\mathcal{X}$ is nonempty by Theorem \ref{cpp-fink}. Furthermore,
$\mathcal{X}$ is closed. In fact, let $B=(b_1,b_2,\dots)$ be in the
closure of $\mathcal{X}$. For every $n$, there exists $A_n\in\mathcal{X}$
in the basic neighborhood $[(b_1,b_2,\dots,b_n)]$. Consider $(A_{n_j})_j$ a
subsequence $(A_{n_j})_j\subseteq(A_n)_n$ such that for every $j$ the
restriction of $R$ to $[A_{n_j}]$ coincides with,
say, $R_m$. To see that $B\in\mathcal{X}$, fix $a$, $b\in[B]$ and $l$
large enough so that $n_l\geq max \{max(supp(a)),max(supp(b))\}$ then
both $a$ and $b$ are member of
$[A_{n_j}]$ for $j\geq l$. Therefore $a\ R \ b\Leftrightarrow a\ R_m\ b$.
This proves that $B\in\mathcal{X}$. By corollary \ref{ana-fink} $\mathcal{X}$
is $\mathcal{H}$--Ramsey.
Consider $A\in \mathcal{H}$ such that $[\emptyset,A]\subseteq
\mathcal{X}$ or $[\emptyset,A]\cap\mathcal{X}=\emptyset$. By theorem
\ref{rcpp-fink}, $[\emptyset,A]\cap\mathcal{X}\not=\emptyset$. Hence
$A\in\mathcal{X}$.
\end{proof}

\section{Final comments}

\subparagraph*{On the stability of Lipschitz functions on $S(c_0)$.} In \cite{Gow}, Gowers used Theorem \ref{gowers} above to prove the following:

\begin{thm}[Gowers; Theorem 6 in \cite{Gow}]\label{Lips}
Let $F : S(c_0) \rightarrow \mathbb{R}$ be an unconditional Lipschitz function. For every real number $\epsilon > 0$, there exists an infinite-dimensional positive block subspace $X$ of $c_0$ such that $sup\{|F(x) - F(y)| : x, y \in S(X)\} < \epsilon$.
\end{thm}

A function satisfying the conclusion of Theorem \ref{Lips} is sometimes called $\epsilon$-\textit{stable}.

\medskip

Given $\epsilon > 0$, let $\delta = \epsilon/2$ and choose an integer $k$ such that  $1/(1+\delta)^{k-1} < \delta$. Let $\Delta_k$ be the collection of functions $h : \mathbb{N} \rightarrow \{1, 1/(1+\delta), \dots, 1/(1+\delta)^{k-1}\}$ which are finitely supported and such that $h(n) = 1$ for some $n$. Then $\Delta_k$ is a $\delta$-net of $PS(c_0)$. There exists a bijective correspondence $\Theta : \Delta_k \rightarrow FIN_k$ defined by $\Theta(h)(n) = k+log_{\delta+1}(h(n))$, if $h(n) \neq 0$;  and $\Theta(h)(n) = 0$, otherwise. 

\medskip

Given an unconditional Lipschitz function $F : S(c_0) \rightarrow \mathbb{R}$ (with Lipschtiz constant equal to 1, without loss of generality), it is possible to use $F$ to define a suitable finite coloring of $\Delta_k$ (see the proof of Theorem 6 in \cite{Gow} for more details) and then use the bijection $\Theta$ to induce a finite coloring $f$  of $FIN_k$ in such a way that any $A\in FIN_k^{[\infty]}$ given by Theorem \ref{gowers} for which $f$ is constant on $[A]$ corresponds to a block basis of $c_0$ whose generated subspace $X$ satisfies the conclusion of Theorem \ref{Lips}.

\medskip

Now, given a  Gowers coideal $\mathcal{H}\subseteq FIN_k^{[\infty]}$, since by definition it satisfies a local version of Theorem \ref{gowers}, it would be interesting to  understand the nature of the family of block bases of $c_0$ (or the family of subspaces of $c_0$) which correspond to $\mathcal{H}$ via the bijection $\Theta$. This could be a means to find more examples of Gowers coideals of block sequences. On the other hand, this could possibly open a highway to transfer the notions of semiselectivity, local Ramseyness, etc, to families of subspaces of $c_0$ and study the relation of such families with the stability of Lipschitz funtions.

\medskip

\subparagraph*{Semiselective Gowers coideals versus stable ordered-union ultrafilters.} 
In \cite{blass}, Blass introduced the \textit{stable ordered-union ultrafilters} on $FIN$ and proved that they satisfy local versions of Hindman's theorem \cite{hindman} and Ramsey's theorem for unions \cite{mill}, the canonical partion property related to Taylor's theorem \cite{taylor} and the infinitary partition property related to a theorem of Milliken's \cite{mill} which states that analytic subsets of $FIN^{[\infty]}$ are Ramsey. On the other hand, for $k = 1$, Theorems \ref{ramsey2} and \ref{semisel-cpp}, and Corollary \ref{ana-fink} in this paper show that semiselective Gowers coideals of block sequences also  satisfy all these properties, besides a local version of a Galvin's theorem for unions (Theorem \ref{galvinlocal} above). Nevertheless, the existence of stable ordered-union ultlafilters cannot be deduced from ZFC alone (see \cite{blass} and \cite{blasshind}), but in this paper we have given examples of semiselective Gowers coideals of block sequences in ZFC.

\bigskip

Finally,

\subparagraph*{Concerning the study of the local Ramsey property in the context of the theory Ramsey spaces.} The results presented in this paper give us a hint on the conditions to be imposed on a family $\mathcal{H}$ of elements of a topological Ramsey space $\mathcal{R}$, in order to obtain a local version (with respect to $\mathcal{H}$) of the \textit{abstract Ellentuck theorem} (see \cite{todo}). Obviously, Theorem \ref{baire-ramsey} above suggests that, besides a property corresponding to semiselectivity, $\mathcal{H}$ must satisfy a local version of the \textit{pigeon hole principle} satisfied by $\mathcal{R}$ (in the case of the topological Ramsey space $FIN_k^{[\infty]}$ the pigeon hole principle considered is precisely Gowers' theorem \cite{Gow} -- Theorem \ref{gowers} above). We refer the reader to \cite{mij2} for partial results on an abstract study of the local Ramsey property.

\bigskip

\begin{flushleft}
\textbf{Acknowledgement.} \\
The authors would like to express their gratitude to Elias Tahhan for his kind revision of a previous draft of this paper and his useful suggestions to improve its presentation, and to Carlos Uzc\'ategui and Carlos Di Prisco for valuable dicussions on the subject matter treated in this work. 
\end{flushleft}


\begin{thebibliography}{99}
\bibitem{blass} Blass, A., \emph{Ultrafilters related to Hindman's
finite-unions theorem and its extensions}, Contemporary Mathematics, \textbf{65}(1987), 90--124.
\bibitem{blasshind} Blass, A., Hindman, N. \emph{On strongly summable ultrafilters and union ultrafilters}, Trans. Amer. Math. Soc. \textbf{304} (1987), no. 1, 83--97.
\bibitem{carsimp} Carlson, T. J, Simpson, S. G. {\em Topological
Ramsey theory}, in Ne\^{s}et\^{r}il, J., R\"{o}dl, {\em Mathematics
of Ramsey Theory} (Eds.), Springer, Berlin, 1990, pp. 172--183.
\bibitem{ellen} Ellentuck, E. {\em A new proof that analytic sets
are Ramsey}, J. Symbolic Logic, \textbf{39}(1974), 163--165.
\bibitem{farah} Farah, I. {\em Semiselective coideals}, Mathematika,
\textbf{45}(1998), 79--103.
\bibitem{galvin} Galvin F.: A generalizition of Ramsey's theorem. Notices of the Amer. Math. Soc. \textbf{15}, 548 (1968).
\bibitem{galpri} Galvin, F., Prikry, K. {\em Borel sets and Ramsey's
theorem}, J. Symbolic Logic, \textbf{38}(1973), 193--198.
\bibitem{Gow} W. T. Gowers, \emph{Lipschitz functions on classical
spaces}, European J. Combin. 13 (1992), 141-151.
\bibitem{hindman} Hindman, N., \emph{The existence of certain ultrafilters
on $\mathbb{N}$ and a conjecture of Graham and Rothschild}, Proc. Amer.
Math. Soc., \textbf{36}(1973), 341--346.
\bibitem{kunen} K. Kunen, \emph{Set theory, an introduction to independence
proofs}, Studies in logic and the foundations of mathematics, vol 102(1980)
Elsevier.
\bibitem{jordi} J. L\'opez--Abad, \emph{Canonical equivalence
relations on nets of $\mathcal{PS}_{C_0}$}, Discrete Math.,
\textbf{307}(2007), 2943--2978.
\bibitem{mathias} Mathias, A. R, {\em Happy families}, Ann. Math.
Logic, \textbf{12}(1977), n1, 59--111.
\bibitem{mij2} Mijares, J. {\em A notion of selective ultrafilter
correspopnding to topological Ramsey spaces}, Math. Log. Q. \textbf{53}(
2007), n3, 255--267.
\bibitem{mill} Milliken, K., \emph{Ramsey's theorem with sums or
unions}, J. Comb. Theory, ser A \textbf{18}(1975), 276--290.
\bibitem{miltodo} A. Miller, {\em Infinite combinatorics and
definibility}, Ann. Pure Appl. Logic  \textbf{41}(1989), 178--203.
\bibitem{Mor} J.C. Morgan, On general theory of point sets II,
Real Anal. Exchange 12(1)(1986/87).
\bibitem{nash} Nash-Williams, C. St. J. A., {\em On
well-quasi-ordering transfinite sequences}, Proc. Cambridge Philo.
Soc., \textbf{61}(1965), 33--39.
\bibitem{pawl} J. Pawlikowski, {\em Parametrized Elletuck theorem},
Topology and its applications \textbf{37}(1990), 65--73.
\bibitem{platek} R. Platek, {\em Eliminating the continuum hipothesis},
J. Symb. Log. \textbf{34}(1969), 219--225.
\bibitem{taylor} A. Taylor, \emph{A canonical partition relation
for finite subsets of $\omega$}, J. Comb. Theory, ser A
\textbf{21}(1976), 137--146.
\bibitem{todo} S. Todorcevic, {\em Introduction to Ramsey spaces},
Princeton University Press, Princeton, New Jersey, 2010.
\bibitem{todo2} S. Todorcevic, {\em High-Dimensional Ramsey theory},
in S.A. Argyros and S. Todorcevic, {\emph Ramsey Methods in
Analysis}, Birkhauser Basel, 1999.
\end{thebibliography}
\end{document}